\pdfoutput=1 
\documentclass[11pt]{article}
\usepackage{amssymb, latexsym, amsmath, theorem}
\numberwithin{equation}{section}

\theoremstyle{plain}
\theorembodyfont{\itshape}
\newtheorem{theorem}{Theorem}[section]
\newtheorem{proposition}[theorem]{Proposition}
\newtheorem{lemma}[theorem]{Lemma}
\newtheorem{corollary}[theorem]{Corollary}

\theorembodyfont{\rmfamily}
\newtheorem{definition}[theorem]{Definition}

\newtheorem{remark}[theorem]{Remark}
\newenvironment{proof}{{\noindent \textbf{Proof}\,\,}}{\hspace*{\fill}$\Box$\medskip}

\title{On the adjacency quantization in the equation modelling the 
Josephson effect}

\author{Alexey Glutsyuk 
\thanks{Laboratoire
J.-V.Poncelet (UMI 2615 du CNRS and the Independent University of Moscow)} 
\thanks{National Research University Higher School of Economics.}  
\thanks{Permanent address: CNRS, Unit\'e de Math\'ematiques
Pures et Appliqu\'ees, M.R., \'Ecole Normale Sup\'erieure de Lyon,
46 all\'ee d'Italie, 69364 Lyon 07, France. Email: 
aglutsyu@ens-lyon.fr}, \ 
Victor Kleptsyn$^*$\thanks{Permanent address: 
CNRS,  Institut de Recherche Math\'ematique de Rennes 
(UMR 6625 CNRS), Campus de Beaulieu, b\^at. 23, 263 Avenue du G\'en\'eral 
Leclerc, CS 74205, 35042 Rennes Cedex, France. Email: victor.kleptsyn@univ-rennes1.fr}
\ Dmitry Filimonov \thanks{Moscow  Institute of Physics and Technology,  
State University}
\thanks{Moscow State Railway University. Email: mityafil@gmail.com}
\ Ilya Schurov \thanks{National Research University Higher School of Economics. Email: ilya@schurov.com}\ \footnotemark[3]\ \thanks{The present paper uses results  obtained 
 under the support of the project No 11-01-0239 ``Invariant manifolds and asymptotic 
 behavior of slow-fast mappings''  within the program  ``The  HSE scientific foundation''  in 2012/2013}
\thanks{Research of all the authors was  partly supported by  RFBR--CNRS joint 
grant  10-01-93115 NTsNIL\_a and RFBR grant  10-01-00739-a. 
Research of the first author (A.Glutsyuk) was partly supported by 
French ANR grant  ANR-08-JCJC-0130-01}
}
\begin{document}
\maketitle

\centerline{\bf\Large 
This is a brief version with plan of the detailed paper written in Russian}

\bigskip

\begin{abstract} In the present paper we investigate two-parametric family of 
non-autonomous ordinary differential equations on the two-torus that model the 
Josephson effect from superconductivity. We study its rotation number as a function 
of parameters and its {\it Arnold tongues} (also called {\it phase locking domains}): 
the  level sets of the rotation number that have non-empty interior. The Arnold 
tongues of the equation under consideration have many non-typical properties: 
the phase locking happens only for integer values of the rotation 
number~\cite{buch2,IRF}; the boundaries of the tongues are given by 
analytic curves~\cite{buch1,buch2004}, the tongues have zero width at the 
intersection points of the latter curves (this yield the adjacency points). Numerical 
experiments and theoretical  investigations  \cite{buch2006,RK} show 
that each Arnold tongue forms an infinite chain of adjacent domains separated by 
adjacency points and going to infinity in asymptotically vertical direction. 
Recent numerical experiments had also shown that for each Arnold tongue 
all its adjacency points lie on one and the same vertical line with 
the integer abscissa equal to the corresponding rotation number.  In the present paper 
we prove this fact for some open domain of the two-parametric families of 
equations under consideration. In the general case we prove a weaker statement: 
the abscissa of each adjacency point is integer; it has the same sign, as the rotation 
number; its modulus is no greater than that of the rotation number. 
The proof is based on the representation of the differential equations under 
consideration as projectivizations of complex linear differential equations on the 
Riemann sphere, see~\cite{buch2004,Foote,IRF}, and the classical theory 
of complex linear equations.
\end{abstract}

\def\td{\mathbb T^2}
\def\rr{\mathbb R}
\def\zz{\mathbb Z}
\def\cc{\mathbb C}
\def\oc{\overline\cc}
\def\diag{\operatorname{diag}}

\section{Introduction}
\subsection{Main results}
We study two-parametric family of ordinary differential equations on the 
torus   $\td=S^1\times S^1$, 
$S^1=\rr\slash2\pi\zz$, with the coordinates $(x,t)$ of the form  

\begin{equation}\dot x=\frac{dx}{dt}=\nu\sin x + a + s \sin t, \ a,\nu,s\in\rr, \ \nu\neq0.\label{jos}\end{equation}
This family of equations, which we will call  {\it class  J equations} for simplicity, 
models the Josephson effect from superconductivity. The value of the parameter $\nu$ 
is supposed to be fixed and non-zero; the variable parameters are $a$ and $s$. 

The period $2\pi$ flow mapping of the equation is a  diffeomorphism of the space circle 
$S^1=S^1\times\{0\}$ 
$$h_{a,s}\colon S^1\to S^1.$$
In the present paper we study its rotation number  $\rho=\rho(a,s)$ as a function of 
the parameters $a$ and $s$. (Scaling convention: the rotation number equals the 
rotation angle divided by $2\pi$.) 
We will call the $a$-axis horizontal  
(and the $a$- coordinate the {\it abscissa}), and the $s$-axis  vertical (and the 
 $s$- coordinate the {\it ordinate}).
\begin{definition} 
The {\it $r$-th Arnold tongue} is the level set 
$\{(a,s)\mid \rho(a,s)=r\}\subset\rr^2$, provided it has a non-empty interior. 
\end{definition}
 
The rotation number of the system~\eqref{jos} has a physical interpretation: 
this is the average tension for a long time interval. The segments of intersection of 
the  horizontal lines with the Arnold tongues correspond to the Shapiro stairs on the 
Wolt-Amper characteristic. Earlier it was shown that 
\begin{itemize}
	\item the Arnold tongues exist only for integer values of the rotation number~\cite{buch2,LSh2009,IRF};
	\item the boundary of each tongue $\rho=r$ consists of two analytic curves that  
		are graphs of functions denoted $a=g^-_r(s)$ and $a=g^+_r(s)$ 
		(see.~\cite{buch1}; Alexey Klimenko independently observed  that this fact 
		follows  immediately from symmetry\footnote{A class J equation has 
		the symmetry   $(x,t)\mapsto (\pi-x,\pi-t)$. This implies that a 
		 parabolic fixed point of the period flow map  can be only a  
		 fixed point $\pm\frac{\pi}2$ of the symmetry, see \cite{RK}. 
		 A proof of an equivalent statement is contained in  
		 \cite[p. 30]{gi}}  of class J equation, see~\cite{RK}); 
	\item each function $g^-_r(s)$, $g^+_r(s)$ has asymptotics of the $r$-th 
	Bessel  function at infinity (found numerically and  justified on the physical level  in~\cite{buch2006}, proved in~\cite{RK}): 
\begin{equation}
\begin{cases}	
& g^-_r(s)=r-\nu J_{r}(-s/\nu)+o(s^{-1/2})\\
&	g^+_r(s)=r+\nu J_r(-s/\nu)+o(s^{-1/2})
\end{cases};
\label{bessel}
\end{equation}

\item thus, each Arnold tongue is an infinite chain of adjacent bounded domains 
that go to infinity in asymptotically vertical direction; the adjacenty 
points of neighbor domains that do not lie on the horizontal axis  $s=0$ are called simply 
{\it adjacencies.}
\end{itemize}

The following fact was found during numerical experiments.
\medskip

{\bf Experimental fact A.} {\it For every fixed $\nu\neq0$ and every $r\in\mathbb Z$ 
all the adjacencies of the $r$-th Arnold tongue lie on the same vertical line  $a=r$.}
\medskip

The main result of the paper is the following theorem that partly proves the above 
experimental fact. 

\begin{theorem}\label{th} The experimental fact A holds for every fixed 
$\nu\neq0$  with $|\nu|\leq 1$. 
For every fixed  $\nu\neq0$ all the adjacencies have integer abscissas. 
The abscissa of each adjacency has the same sign, as the corresponding rotation 
number, and its module is no greater than that of the rotation number. 
The adjacencies corresponding to zero rotation number are exactly those lying on the 
axis $a=0$.   
\end{theorem}

\begin{corollary} \label{slkr} For every $\nu\neq0$ and $r\in\mathbb Z$ there exists a   
  $M=M(\nu,r)>0$ such that all the adjacencies of the $r$-th Arnold tongue 
with ordinates having modulus greater than  $M$ lie on the line $a=r$.
\end{corollary}

Corollary \ref{slkr} follows from the integrality of the abscissas of the adjacencies   
(Theorem  
\ref{th}) and  the asymptotic formula  (\ref{bessel}) for the boundary of the 
Arnold tongues: the latter formula implies that the points of a given   
 tongue with ordinates large enough  lie in the 1-neighborhood of the line $a=r$.  

\begin{remark} \label{rperem} 
It is known that for every $r\in\mathbb Z\setminus0$ the Arnold tongue $\rho(a,s)=r$ 
intersects the horizontal axis  $s=0$ exactly at the point with the abscissa 
 $\sqrt{r^2+\nu^2}$ (see \cite[Corollary 3]{buch1} and~\cite{IRF}).  This is the 
 point of adjacency of neighbor components of the tongue that we will call 
 {\it queer adjacency.} 
\end{remark}

\subsection{Sketch-proof and plan of the detailed paper}

The proof of Theorem \ref{th} is given in Section 3. It is based on the 
interpretation  of a family of  class J equations as a family of projectivizations 
of linear ordinary differential equations on the Riemann sphere and  classical theory 
of complex linear equatons. The above-mentioned interpretation is given by the 
next proposition stated and proved in Subsection 2.2. 
Its reformulation in some other terms was 
obtained  in  ~\cite{buch2004,Foote,IRF}.

\begin{proposition} \label{prop-change} The variable change 
$p=e^{ix}, \ \tau=e^{it}$ 
reduces the family of class J equations to the following family of Riccati equations:  
\begin{equation}\frac{dp}{d\tau}=\frac1{\tau^2}\left(\nu(1-p^2)\frac{i\tau}2+(a\tau+\frac{is(1-\tau^2)}2)p\right).\label{ricc}\end{equation}
The latter is the projectivization of the following family of linear ordinary differential 
equations via the variable change $p=\frac{z_2}{z_1}$: 
\begin{equation}\dot z=\frac{A(\tau)}{\tau^2}z, \ z=(z_1,z_2)\in\cc^2,  \ 
A(\tau)= \left(\begin{matrix}& 0 & \frac{i\nu\tau}2 \\
& \frac{i\nu\tau}2 & \frac{is}2(1-\tau^2)+a\tau\end{matrix}\right).
\label{lin}\end{equation}
\end{proposition}

 \begin{definition} A nondegenerate linear operator in a vector space is  
 {\it projectively identical (trivial),} if its projectivization is identity as the induced map of the 
 projective space. 
 \end{definition}
 
 The integrality of the abscissas of adjacencies for every $\nu\neq0$ is proved in 
 Subsection 3.1. The linear equations (\ref{lin}) 
 corresponding to class J equations have two 
 irregular nonresonant Poincare rank 1 singular points: 0 and  $\infty$ (see 
 Subsection 2.2).

 \begin{proposition} \label{jos-mon-id} 
 The adjacencies correspond to those parameter values, for which  linear equation (\ref{lin}) has a projectively identical monodromy operator 
 along the positive circuit around zero.
 \end{proposition}
 
 \begin{proof} {\bf (sketch).}
 The adjacencies correspond to those parameter values, for which the period $2\pi$ 
 flow map of class J equation is the identity. This easily follows from  elementary 
 properties of the rotation number (monotonicity in $a$ and continuity in both 
 parameters). The complexification of the latter period flow map is 
 conjugate to the projectivization of the monodromy of linear equation (\ref{lin}), by  
Proposition \ref{prop-change}. This implies Proposition \ref{jos-mon-id}. 
\end{proof}
 
 \begin{lemma} \label{lemon} The monodromy matrix of 
 equation (\ref{lin}) is written in appropriate basis of solutions as a product 
 \begin{equation}
 M=\operatorname{diag}(1,e^{2\pi i a}) S_1S_0,\label{decomp}\end{equation} 
 where the matrices  $S_0$ ($S_1$)
are unipotent and respectively lower (upper) triangular.
 \end{lemma}
 
 \begin{proof} {\bf (sketch).} 
 The classical theory of linear equations at nonresonant irregular singularities 
 \cite{2, 12, bjl,jlp,sib} implies that
  the monodromy matrix of a linear equation at a Poincar\'e rank 1 irregular 
nonresonant singularity  in appropriate basis of solutions  
is equal to the product of three matrices: 

- a diagonal matrix, which is the  
monodromy matrix of the formal normal form of the equation;
 
- two additional unipotent matrices $S_1$ and $S_0$, which are inverse to the so-called 
{\it Stokes matrices.}

The above product decomposition is explicitly given in \cite[p. 35]{12}.

The formal normal form of equation (\ref{lin}) is 
 \begin{equation}\begin{cases} & \dot w_1=0 \\
& \dot w_2=\frac1{\tau^2}(\frac{is}2+a\tau)w_2\end{cases}.\label{norm-jos}\end{equation}
Its monodromy matrix is $\operatorname{diag}(1,e^{2\pi i a})$. This together with the 
above product decomposition proves the lemma.
\end{proof}

\begin{proof} {\bf of the integrality of the abscissas of the adjacencies.} The 
projective triviality of the monodromy of equation (\ref{lin}) is equivalent to the 
triviality of all the operators in the product (\ref{decomp}). This follows from the 
unipotence and  lower (upper) triangularity of the matrices $S_0$ ($S_1$). The 
first operator $\operatorname{diag}(1,e^{2\pi i a})$ is trivial, if and only if 
$a\in\mathbb Z$. Therefore, the abscisses of the adjacencies are integer.
\end{proof}

The experimental fact A  for $|\nu|\leq1$ is proved in Subsection 3.2. An additional 
elementary differential inequality 
(a precision of lemma 4 from \cite{buch1}) 
 shows that for  $|\nu|\leq1$ the complement of the $r$-th Arnold tongue to the  horizontal axis  
 $s=0$ lies strictly between the lines $a=r\pm1$. Thus, all its adjacencies should lie on 
 the  line  $a=r$. 

The general case, when  $\nu$ is arbitrary, is treated in Subsection 3.3. It is easy to 
show that the abscissa of an adjacency has the same sign, as the corresponding 
rotation number.  Using the Argument Principle for appropriate solutions of 
Riccati equations (\ref{ricc}), we  show that the modulus of the abscissa of 
each adjacency is no greater than the modulus of the corresponding 
rotation number. This proves the theorem. 

A preparatory material (rotation number, linear equations, 
Stokes operators and monodromy) is contained in Section 2. 

\section{Acknowledgements}
The authors are grateful to V.M.Buchstaber, E.Ghys, Yu.S.Ilyashenko, A.V.Klimenko 
and O.L.Romaskevich for helpful discussions. 

\end{document}


\maketitle

\begin{abstract} В статье исследуется двупараметрическое семейство неавтономных
	обыкновенных дифференциальных уравнений на торе, моделирующее эффект
	Джозефсона из физики сверхпроводников.  
Изучается его число вращения как
	функция от параметров и {\it языки Арнольда} (иначе называемые {\it областями фазового захвата}): 
	множества уровня числа вращения, имеющие непустую
	внутренность. Языки рассматриваемого уравнения обладают рядом нетипичных
	свойств: фазовый захват происходит только для целочисленных значений числа
	вращения~\cite{buch2,IRF}; границы языков задаются аналитическими
	кривыми~\cite{buch2004, buch1}, в точках пересечения которых ширина языка равна нулю (образуются перемычки). Численные эксперименты и теоретические
	исследования  (\cite{buch2006}, \cite{RK}) 
	показывают, что каждый язык Арнольда образует
	бесконечную цепочку примыкающих друг к другу областей, разделенных перемычками и уходящих на
	бесконечность в асимптотически вертикальном направлении. Недавно в ходе численных экспериментов было 
	также обнаружено, что для каждого языка Арнольда все его
	перемычки ложатся на одну и ту же вертикальную прямую с целочисленной
	абсциссой, равной соответствующему числу вращения. В статье приведено
	доказательство этого факта для некоторого открытого множества рассматриваемых двупараметрических семейств уравнений.  В общем случае 
	доказано более слабое 
 утверждение: абсцисса каждой перемычки целочисленна, имеет тот же знак, 
 что и число вращения, и по модулю не превосходит 
 числа вращения. 
Доказательство  основано на представлении рассматриваемых дифференциальных уравнений как проективизаций линейных дифференциальных уравнений на сфере Римана~\cite{buch2004,IRF,Foote} и
	классической теории линейных уравнений с комплексным временем.
\end{abstract}

\def\td{\mathbb T^2}
\def\rr{\mathbb R}
\def\zz{\mathbb Z}
\def\cc{\mathbb C}
\def\oc{\overline\cc}
\def\diag{\operatorname{diag}}

\section{Введение}
\subsection{Основные результаты}
В статье рассматривается  семейство 
обыкновенных  дифференциальных 
уравнений на торе $\td=S^1\times S^1$, 
$S^1=\rr\slash2\pi\zz$, с координатами $(x,t)$, имеющее вид 

\begin{equation}\dot x=\frac{dx}{dt}=\nu\sin x + a + s \sin t, \ a,\nu,s\in\rr, \ \nu\neq0.\label{jos}\end{equation}
Это семейство уравнений, которые мы будет для краткости называть 
{\it уравнениями класса Д,} 
моделирует эффект Джозефсона из физики сверхпроводников. 
Значение параметра $\nu$ фиксируется произвольным, не равным нулю. 

Отображение потока уравнения за период $2\pi$ есть диффеоморфизм пространственной окружности $S^1=S^1\times\{0\}$ 
$$h_{a,s}\colon S^1\to S^1.$$
В настоящей статье исследуется его число вращения $\rho=\rho(a,s)$ как функция 
от параметров $a$ и $s$. (Соглашение о нормировке: 
 число вращения поворота равно углу поворота, 
деленному на $2\pi$.) Будем называть координатную ось $a$ 
горизонтальной (и $a$-координату~--- {\it абсциссой}), а 
ось $s$~--- вертикальной (и $s$-координату~--- {\it ординатой}).
\begin{definition} 
Назовём {\it $r$-м языком Арнольда} множество уровня 
$\{(a,s)\mid \rho(a,s)=r\}\subset\rr^2$, если оно имеет непустую внутренность. 
\end{definition}
 
Число вращения системы~\eqref{jos} имеет физическую интерпретацию: это среднее напряжение за длительный промежуток времени. Отрезки, по которым языки Арнольда пересекают горизонтальные прямые, соответствуют ступенькам Шапиро на вольт-амперной характеристике. Ранее было показано, что
\begin{itemize}
	\item языки Арнольда существуют только для целых значений числа вращения~\cite{buch2,IRF,LSh2009};
	\item граница каждого языка $\rho=r$ состоит из двух аналитических кривых, 
		являющихся графиками функций, обозначаемых через $a=g^-_r(s)$ и $a=g^+_r(s)$ (см.~\cite{buch1}; А. В. Клименко независимо заметил, что этот факт мгновенно следует из соображений симметрии\footnote{Уравнение 
		класса Д имеет симметрию $(x,t)\mapsto (\pi-x,\pi-t)$, в силу которой параболическими неподвижными точками  отображения потока за период могут быть только неподвижные точки $\pm\frac{\pi}2$ симметрии, см. \cite{RK}. Доказательство эквивалентного утверждения 
		содержится в \cite[стр. 30]{gi}}  уравнения 
		класса Д, см.~\cite{RK}); 
	\item каждая функция $g^-_r(s)$, $g^+_r(s)$ имеет асимптотику $r$-й функции Бесселя 
		на бесконечности (численно обнаружено и обосновано на физическом уровне строгости в ~\cite{buch2006}, доказано в 
		 ~\cite{RK}): 
\begin{equation}
\begin{cases}	
& g^-_r(s)=r-\nu J_{r}(-s/\nu)+o(s^{-1/2})\\
&	g^+_r(s)=r+\nu J_r(-s/\nu)+o(s^{-1/2})
\end{cases} 
\label{bessel}
\end{equation}

\item тем самым, каждый язык Арнольда есть бесконечная 
цепочка примыкающих друг к другу ограниченных областей, 
уходящих на бесконечность в асимптотически вертикальном направлении;  
точки примыкания соседних областей,  не лежащие на оси $s=0$,  называются 
{\it перемычками.}
\end{itemize}

В ходе численных экспериментов был обнаружен следующий 

\medskip

{\bf Экспериментальный факт А.} {\it Для любого 
фиксированного $\nu\neq0$ и любого $r\in\mathbb Z$ 
все перемычки  $r$- го языка Арнольда лежат 
на одной и той же вертикальной прямой $a=r$.}
\medskip

Основным результатом настоящей статьи является следующая теорема, которая 
частично доказывает вышеупомянутый экспериментальный факт. 

\begin{theorem}\label{th} Экспериментальный факт А верен 
при любом фиксированном $\nu\neq0$  с $|\nu|\leq 1$. 
При любом фиксированном $\nu\neq0$ 
все перемычки имеют  целочисленные абсциссы. У каждой перемычки 
абсцисса   имеет тот же знак, что и соответствующее число 
вращения, и не превосходит числа вращения по модулю. Перемычки, 
отвечающие нулевому числу вращения, суть в точности перемычки, 
лежащие на оси $a=0$.  
\end{theorem}

\begin{corollary} \label{slkr} Для любых $\nu\neq0$ и $r\in\mathbb Z$ существует  
  $M=M(\nu,r)>0$, такое что все перемычки $r$- го 
языка Арнольда с ординатами, по модулю большими $M$, лежат на 
прямой $a=r$.
\end{corollary}

Следствие \ref{slkr} вытекает из целочисленности абсцисс перемычек (теорема 
\ref{th}) и асимптотической формулы (\ref{bessel}) для границы языка Арнольда, 
в силу которой точки $r$-го языка с достаточно большими ординатами по модулю 
лежат в 1-окрестности прямой $a=r$.  

\begin{remark} \label{rperem} 
Известно, что при любом $r\in\mathbb Z\setminus0$ язык Арнольда $\rho(a,s)=r$ 
пересекает горизонтальную ось $s=0$ ровно в одной точке 
с абсциссой $\sqrt{r^2+\nu^2}$ (см.~\cite{IRF} и \cite[следствие 3]{buch1}).  
Это --- точка примыкания соседних компонент  
внутренности языка, которую мы будем называть {\it ложной перемычкой.} 
\end{remark}

\subsection{Идея доказательства и план статьи}

Доказательство  теоремы \ref{th} представлено в разделе 3. 
Оно основано на представлении семейства уравнений класса Д как 
семейства проективизаций линейных дифференциальных уравнений на 
сфере Римана (полученном в работах различных авторов, см. также раздел 
2.2 ниже) 
и классической теории линейных уравнений с комплексным временем.

 \begin{definition} Будем называть невырожденный линейный оператор 
 в линейном пространстве 
 {\it проективно тождественным,} если он имеет тождественную проективизацию 
 как индуцированное отображение проективного пространства. 
 \end{definition}
 
 Целочисленность абсцисс перемычек при любом $\nu\neq0$ доказывается в 
 разделе 3.1. 
Линейные уравнения, соответствующие уравнениям класса Д, имеют 
две нерезонансные иррегулярные особые точки 0 и $\infty$ ранга Пуанкаре 1 
 на сфере Римана (см. раздел 2.2). 
Перемычки отвечают тем значениям параметров, при которых 
монодромия линейного уравнения проективно тождественна. 
Оказывается, что это имеет место 
тогда и только тогда, когда росток рассматриваемого линейного уравнения в 
иррегулярной особой точке 0 приводится 
аналитической заменой переменных к его формальной нормальной форме 
(прямой сумме одномерных уравнений), и монодромия последней 
проективно тождественна. Это выводится из 
классических результатов об аналитической классификации 
ростков линейных уравнений в нерезонансных иррегулярных особых точках. 
Матрица вычета нормальной формы имеет единственное ненулевое 
собственное значение, равное абсциссе перемычки.  
Отсюда следует целочисленность абсциссы перемычки. 

Экспериментальный факт А  при $|\nu|\leq1$ доказывается в разделе 3.2. 
Дополнительное элементарное дифференциальное неравенство 
(уточнение леммы 4 из \cite{buch1}) 
 показывает, что при $|\nu|\leq1$ дополнение $r$-го языка Арнольда до горизонтальной оси 
 $s=0$ лежит строго между прямыми $a=r\pm1$. Тем самым, 
все его  перемычки должны лежать на прямой $a=r$. 

Общий случай, когда $\nu$ --- 
произвольно, разбирается в разделе 3.3. Легко 
показать, что абсцисса перемычки имеет тот же знак, что и соответствующее 
число вращения. Дополнительные рассуждения о соответствующих 
уравнениях Риккати, использующие  принцип аргумента для комплексных 
решений, доказывают, что модуль абсциссы каждой перемычки не превосходит 
модуля числа вращения. Это докажет теорему. 

Подготовительный материал (число вращения, линейные уравнения, 
операторы Стокса и монодромия) содержится в разделе 2. 

\section{Предварительные сведения}
\subsection{Число вращения потока на торе и перемычки.}

Рассмотрим поток на торе $\td=S^1\times S^1=\rr^2\slash2\pi\zz^2$ 
с координатами $(x,t)$, 
заданный неавтономным дифференциальным уравнением с гладкой правой 
частью
\begin{equation}\dot x=\frac{dx}{dt}=f(x,t).\label{potok}\end{equation}
Преобразование потока за время $t$ есть диффеоморфизм пространственной 
окружности, который мы обозначим $h_t:S^1\to S^1$. 
Рассмотрим универсальную  накрывающую над пространственной окружностью: 
$$\rr\to S^1=\rr\slash2\pi\zz.$$
Преобразования потока уравнения (\ref{potok}) поднимаются 
на универсальную накрывающую и индуцируют диффеоморфизмы
$$H_{q,t}:\rr=\rr\times\{ q\}\to\rr=\rr\times \{ q+t\}, \ H_{q,0}=Id.$$
Напомним, что для любых $(x,q)\in\rr\times S^1$  предел 
\begin{equation}
\rho=\lim_{n\to+\infty}\frac1nH_{q,2\pi n}(x)\in\rr\label{rot}\end{equation}
существует, не зависит ни от $q$, ни от $x$  и называется 
{\it числом вращения потока уравнения (\ref{potok})} (см. напр.~\cite[стр. 124]{Arn}).

Рассмотрим теперь (произвольное) аналитическое семейство уравнений 
\begin{equation}v_{a,s}: \ \dot x=g(x,t,s)+a, \ a,s\in\rr.\label{sem-a}\end{equation}

\begin{proposition} \label{monot} Число вращения $\rho=\rho(a,s)$ потока (\ref{sem-a}) 
есть непрерывная функция от параметров $(a,s)$, которая не убывает 
как  функция от $a$. Если при некотором 
значении параметров $(a_0,s_0)$  преобразование потока 
$h_{2\pi}=h_{a_0,s_0,2\pi}:S^1\times\{0\}\to S^1\times \{0\}$ уравнения (\ref{sem-a}) 
имеет неподвижную точку, то число 
вращения $\rho(a_0,s_0)$ целочисленно. Если в последнем случае, дополнительно,  преобразование потока  $h_{2\pi}$ 
не является тождественным, то существует такой интервал 
$I$ с концом $a_0$, что 

- для любого $a\in I$ преобразование $h_{a,s_0,2\pi}$ 
имеет хотя бы одну  неподвижную точку; 

- для любого $a\in I$  потоки всех 
векторных полей, достаточно $C^1$- близких к $v_{a,s_0}$, имеют то же число 
вращения $\rho(a_0,s_0)$. 
\end{proposition}

\begin{proof} Первое и второе утверждения предложения 
содержатся в~\cite[стр. 131--133]{Arn}. Докажем третье утверждение предложения. 
Пусть преобразование потока $h=h_{a_0,s_0,2\pi}$ не является тождественным и 
имеет неподвижную точку $O\in S^1$. Тогда последняя --- изолированная 
неподвижная точка, в силу аналитичности. Пусть  
 $J=[P,O]\subset S^1$ --- отрезок с концом $O$ слева от $O$, не содержащий других 
 неподвижных точек. Тогда $h(J)\neq J$: либо $h(J)\subset J$, либо $h(J)\supset J$. 
 Без ограничения общности мы рассмотрим только первый случай: 
 $h_{a,s_0,2\pi}(J)\subset J$ при $a=a_0$. 
Существует интервал $I=(c,a_0)\subset\mathbb R$, такой что 
последнее включение выполнено для любого $a\in I$ и является 
строгим, $h_{a,s_0,2\pi}(J)\Subset J$, так как для любых 
фиксированных $s$, $t$ и $x$ функция $a\mapsto h_{a,s,t}(x)$ 
 непрерывна и строго возрастает. Тем самым, для любого $a\in I$ отображение 
 $h_{a,s_0,2\pi}$ переводит отрезок $J$ внутрь себя, а значит, имеет там неподвижную 
 точку, и следовательно, целое число вращения, равное $\rho(a_0,s_0)$. 
 Последнее включение выживает при любых $C^0$- малых возмущениях 
 отображения $h_{a,s,2\pi}$.  
 Это вместе с непрерывностью числа вращения доказывает предложение \ref{monot}.
 \end{proof}

\def\wt#1{\widetilde#1}

Рассмотрим теперь семейство (\ref{jos}) уравнений класса Д. 
Для каждых значений параметров $(a,s)$ обозначим 
через 
\begin{equation}h_{a,s}=h_{a,s,2\pi}:S^1\times \{0\}\to S^1\times \{0\} 
\label{holonomy}\end{equation}
преобразование потока соответствующего уравнения за период $2\pi$. Пусть  
 \begin{equation}\rho(a,s) \text{  обозначает соответствующее число вращения.}
 \label{rho-as}\end{equation}

\begin{proposition} \label{perem-id} Число вращения $\rho(a,s)$ есть непрерывная функция от параметров ($a,s)$, 
которая не убывает как функция от $a$. Перемычки языков Арнольда 
и ложные перемычки 
(см. введение)  отвечают в точности тем значениям параметров, при которых 
соответствующее преобразование потока $h_{a,s}$ тождественно.
\end{proposition}
\begin{proof}
Первое утверждение предложения \ref{perem-id} следует 
из первого утверждения предложения \ref{monot}. 
Докажем второе утверждение. 
Предположим противное: для одной из перемычек $(a_0,s_0)$ 
(настоящей или ложной) 
 соответствующее 
преобразование потока не является тождественным. Оно имеет целочисленное 
число вращения и тем самым, неподвижную точку. Следовательно, имеется 
интервал $I\subset\rr$, примыкающий к $a_0$, такой что для любого $a\in I$ 
преобразование $h_{a,s_0}$ имеет  неподвижную точку 
и интервал $I\times\{ s_0\}\subset\{ s=s_0\}$ 
 лежит во внутренности языка Арнольда 
$\rho(a,s)=\rho(a_0,s_0)$ (предложение \ref{monot}).  В то же время,  прямая  
$s=s_0$ не может пересекать внутренности языка. Действительно, 
в противном случае, близкая  ей 
параллельная  прямая пересекла бы, как минимум, 
две различные компоненты языка (по определению перемычки), что 
противоречило бы монотонности числа вращения как функции от $a$. 
Полученное противоречие доказывает тождественность преобразования потока. 
И обратно, пусть преобразование потока тождественно. Тогда рассматриваемая 
точка $(a_0,s_0)$ принадлежит целочисленному языку Арнольда. Она не может 
лежать в его внутренности. Действительно, при увеличении (уменьшении) 
параметра  $a$ при фиксированном $s=s_0$  образы всех точек $x\in S^1$ 
под действием преобразования потока движутся в  
положительную (отрицательную) сторону. В частности, у возмущенного 
отображения потока не будет неподвижных точек (в силу тождественности 
невозмущенного). Тем самым, точка $(a_0,s_0)$ горизонтальной прямой 
$\rr\times \{ s_0\}$ является изолированной точкой пересечения этой прямой с 
языком Арнольда $\rho=\rho(a_0,s_0)$. Следовательно, она должна быть 
перемычкой (настоящей или ложной). Предложение доказано. 
\end{proof}

\subsection{Сведение уравнений класса Д к уравнениям Риккати}
Результаты этого раздела в несколько иных терминах ранее были получены в работах~\cite{buch2004,IRF,Foote}.

\begin{proposition} \label{prop-change} Замена переменных
\begin{equation}p=e^{ix}, \ \tau=e^{it}\label{change1}\end{equation}
сводит семейство уравнений класса Д к семейству уравнений 
вида Риккати:
\begin{equation}\frac{dp}{d\tau}=\frac1{\tau^2}\left(\nu(1-p^2)\frac{i\tau}2+(a\tau+\frac{is(1-\tau^2)}2)p\right).\label{ricc}\end{equation}
Последнее получается из следующего семейства линейных обыкновенных 
дифференциальных уравнений проективизацией: 
\begin{equation}\dot z=\frac{A(\tau)}{\tau^2}z, \ z=(z_1,z_2)\in\cc^2,  \ 
A(\tau)= \left(\begin{matrix}& 0 & \frac{i\nu\tau}2 \\
& \frac{i\nu\tau}2 & \frac{is}2(1-\tau^2)+a\tau\end{matrix}\right); \ \ p=\frac{z_2}{z_1}.
\label{lin}\end{equation}
\end{proposition}
\begin{proof} Подставляя замену (\ref{change1}) в уравнение (\ref{jos}), получаем 
$$\sin x =\frac1{2i}(p-p^{-1}), \ \sin t=\frac1{2i}(\tau-\tau^{-1}),$$
$$\dot p=\frac{dp}{dt}=ip\dot x=\frac{\nu p}2(p-p^{-1})+iap+\frac s2(\tau-\tau^{-1})p, \ 
\dot\tau=\frac{d\tau}{dt}=i\tau.$$
Следовательно,
$$ \frac{dp}{d\tau}=\frac{\dot p}{\dot\tau}=\frac1{\tau^2}(\frac{i\nu\tau}2(1-p^2)+
ap\tau+\frac{is}2(1-\tau^2)p),$$
и тем самым, получаем  уравнение (\ref{ricc}). Последнее есть уравнение вида 
Риккати, и тем самым, является проективизацией линейного уравнения 
\begin{equation}\dot z=\frac{dz}{d\tau}=\frac{B(\tau)}{\tau^2}z, \ z=(z_1,z_2)\in\cc^2 , \ 
p=\frac{z_2}{z_1}.\label{lineopr}\end{equation}
Уравнение (\ref{lineopr}) однозначно определено с точностью до нормировки, 
так как соответствующая вектор-функция $z(\tau)$ определяется 
решением $p(\tau)$ уравнения Риккати однозначно с точностью до умножения 
на скалярную функцию от $\tau$. Найдем коэффициенты уравнения 
(\ref{lineopr}).  Подставляя его в формулу $p=\frac{z_2}{z_1}$ и дифференцируя 
последнюю, получим 
$$\frac{dp}{d\tau}=\frac1{\tau^2}((B_{22}(\tau)-B_{11}(\tau))p+B_{21}(\tau)-B_{12}(\tau)p^2).$$
Сравнение последнего уравнения с уравнением (\ref{ricc}) дает 
$$B_{22}(\tau)-B_{11}(\tau)=a\tau+\frac{is}2(1-\tau^2), \ 
B_{21}(\tau)=B_{12}(\tau)=\frac{i\nu\tau}2.$$
Всякая матричнозначная функция $B=(B_{ij})(\tau)$, удовлетворяющая 
предыдущим равенствам, 
определяет линейное уравнение, отвечающее уравнению Риккати (\ref{ricc}). 
Функция $B_{11}(\tau)$ может быть выбрана произвольной, и её выбор  
определяет матрицу $B(\tau)$ однозначно. Положив $B_{11}\equiv0$, 
получим матричнозначную функцию $A(\tau)$ из (\ref{lin}). Предложение доказано. 
\end{proof}
\def\la{\lambda}
\subsection{Иррегулярные особые точки линейных дифференциальных уравнений:  
явление Стокса и монодромия}
Все результаты, представленные в  настоящем разделе, являются классическими и содержатся в  \cite{2, 12, bjl, jlp, sib}.

Рассмотрим росток линейного обыкновенного дифференциального уравнения 
\begin{equation}\dot z=\frac{B(\tau)}{\tau^2}z, \ z=(z_1,z_2)\in\cc^2\label{linirr}
\end{equation}
в окрестности нерезонансной иррегулярной особой точки $\tau=0$ ранга 1. 
По определению, это в точности означает, что  $B(\tau)$ есть голоморфная 
в нуле $2\times 2$ матричнозначная функция, и матрица $B(0)$ имеет 
различные собственные значения $\la_1$, $\la_2$. Без ограничения общности будем считать матрицу $B(0)$ диагональной, 
$$B(0)=\diag(\la_1,\la_2), \  \la_2-\la_1\in i\rr_+.$$
Этого можно добиться, применив линейные постоянные замены переменных  
$z$ и $\tau$. 

\begin{definition}  Два ростка уравнений вида ~\eqref{linirr} \emph{ аналитически
(формально) эквивалентны,} если существует замена $z=H(\tau)w$ переменной  
$z$, где $H(\tau)$ --- голоморфная обратимая матричнозначная функция
(соответственно, формальный обрaтимый степенной  ряд по $\tau$ с матричными 
коэффициентами), преобразующая одно уравнение в другое. 
\end{definition}
Аналитическая классификация иррегулярных нерезонансных ростков 
линейных уравнений  (\ref{linirr}) и результаты, представленные в настоящем 
разделе, были получены в классических работах 
\cite{bjl,jlp,sib} а  также представлены в \cite{2, 12}.
 Оказывается, что всякий росток уравнения (\ref{linirr}) 
формально эквивалентен единственной прямой сумме одномерных 
уравнений вида 
\begin{equation}\begin{cases} &\dot w_1=\frac{(b_{10}+b_{11}\tau)}{\tau^2}w_1\\
&\dot w_2=\frac{(b_{20}+b_{21}\tau)}{\tau^2}w_2.\end{cases}, \ \ b_{20}-b_{10}
\in i\rr_+, 
\label{linform}\end{equation}
называемой {\it формальной нормальной формой;} здесь $b_{j0}=\lambda_j$.  
Соответствующий нормализующий ряд  
$H(\tau)$, $z=H(\tau)w$, единственен с 
точностью до умножения справа на постоянную матрицу. Как правило, 
нормализующий ряд  расходится. В то же время, существует покрытие 
проколотой окрестности нуля двумя секторами $S_0$ и $S_1$ с вершиной в нуле 
 на комплексной прямой $\tau$, обладающее следующим свойством. 
 Над каждым сектором $S_j$ существует 
единственная замена переменных
$z=H_j(\tau)w$, которая преобразует \eqref{linirr} в
\eqref{linform}, где $H_j(\tau)$ --- аналитическая обратимая матричнозначная
функция на $S_j$,  $C^\infty$-гладко продолжающаяся на замыкание 
$\overline S_j$ сектора, так что ее асимптотический ряд Тейлора в
$0$ существует и совпадает с данным нормализующим рядом. Предыдущее утверждение о
существовании и единственности секториальной нормализации выполнено в
любом хорошем секторе (см. следующее определение); покрытие
состоит из хороших секторов. 

\begin{definition} \label{def:1.2}
Назовем сектор в $\mathbb C$ с вершиной в $0$ \emph{хорошим}, если он
содержит только одну вещественную полуось $\mathbb R_{\pm}$, и его
замыкание не содержит другую (см. рис.~\ref{fig:1}).
\end{definition}
 Будем считать, 
что сектор $S_0$ содержит вещественную положительную полуось, а сектор 
$S_1$ --- отрицательную, см. рис. \ref{fig:1}. Положим $S_2=S_0$. 

Стандартное разложение 
нормальной
формы \eqref{linform} в прямую сумму одномерных уравнений определяет
канонический базис в пространстве её решений (единственным образом,  
с точностью до умножения базисных функций на константы) с диагональной
фундаментальной матрицей. Обозначим последнюю фундаментальную матрицу
через 
$$W(\tau)=\diag(w_1,w_2).$$
Вместе с нормализующими заменами $H_j$ в $S_j$,  она определяет
канонические базисы $(f_{j1},f_{j2})$ в пространствах решений уравнения 
\eqref{linirr} в секторах $S_j$ с фундаментальными матрицами
\begin{equation}\label{eq:1.5}
Z^j(\tau)=H_j(\tau)W(\tau), \quad j=0,1,
\end{equation}
где для любого $j=0,1$ 
ветвь ``с индексом $j+1$'' фундаментальной матрицы $W(\tau)$ на $S_{j+1}$
получается из её ветви ``с индексом $j$'' на $S_j$ с помощью аналитического продолжения против часовой стрелки. Ветвь фундаментальной 
матрицы $W$ ``с индексом
$2$'' на $S_2=S_0$ получается из ветви ``с индексом $0$'' 
умножением 
справа на матрицу монодромии формальной нормальной формы 
\eqref{linform}. В связной компоненте пересечения $S_j\cap S_{j+1}$
имеются два канонических базиса решений, приходящих из секторов $S_j$ и
$S_{j+1}$. Как правило, они не совпадают. Переход между ними
определяется при помощи постоянной матрицы $C_j$: 
\begin{equation}\label{eq:1.6}
Z^{j+1}(\tau)=Z^j(\tau)C_j.
\end{equation}
Операторы переходов (матрицы $C_j$) называются \emph{операторами
(матрицами) Стокса}  (см. вышеупомянутую литературу). Нетривиальность операторов Стокса доставляет препятствие к 
аналитической эквивалентности уравнения \eqref{linirr} 
его формальной нормальной форме \eqref{linform} и называется 
{\it явлением Стокса.}

\begin{figure}[ht]
  \begin{center}
   \includegraphics{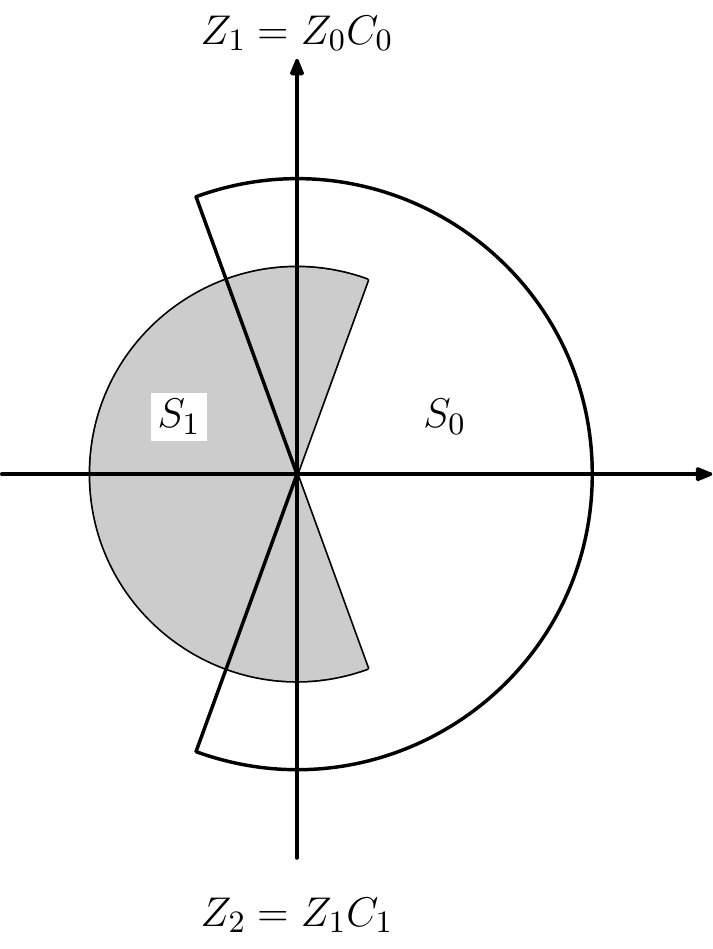}
    \caption{Сектора $S_0$, $S_1$ и операторы Стокса $C_0$ и $C_1$ над их 
    пересечениями}
    \label{fig:1}
  \end{center}
\end{figure}

\begin{remark} \label{rem:1.6}
Матрицы Стокса \eqref{eq:1.6} однозначно определены с точностью до
одновременного сопряжения с помощью одной и той же диагональной матрицы. 
{\it Матрицы Стокса унипотентны.} Матрица $C_0$, отвечающая верхней 
компоненте пересечения секторов, нижнетреугольна, а матрица  $C_1$ --- 
верхнетреугольна: 
\begin{equation}C_0=\left(\begin{matrix} & 1 & 0 \\
& c_0 & 1\end{matrix}\right), \ C_1=\left(\begin{matrix} & 1 & c_1 \\
& 0 & 1\end{matrix}\right).\label{triang}
\end{equation}
\end{remark}

\begin{theorem} \label{th-aneq} \cite{2, 12, bjl,jlp,sib}
Уравнение (\ref{linirr}) аналитически эквивалентно своей 
формальной нормальной форме (\ref{linform}), если и только если 
оно имеет тривиальные операторы Стокса. Два уравнения (\ref{linirr}) аналитически 
эквивалентны, если и только если они имеют одинаковые формальные нормальные формы и наборы матриц Стокса (с точностью до 
вышеупомянутого одновременного сопряжения с помощью диагональной 
матрицы). 
\end{theorem}

Пусть $\gamma:[0,1]\to\cc\setminus0$ --- замкнутый 
путь на проколотой прямой $\tau$, 
делающий один обход вокруг нуля против часовой стрелки. 
Напомним, что {\it оператор монодромии} уравнения (\ref{linirr}) --- это линейный 
оператор, действующий в пространстве локальных решений в окрестности 
точки $\gamma(0)$ и переводящий каждое решение в результат его 
аналитического продолжения вдоль пути $\gamma$.

\begin{lemma} \label{lemon} \cite[стр. 35]{12} Матрицы монодромии  
уравнения (\ref{linirr}) и его формальной нормальной формы (\ref{linform}), обозначаемые, соответственно, $M$ и $M_N$, и введенные выше матрицы 
Стокса $C_0$ и $C_1$ удовлетворяют равенству
\begin{equation}M=M_NC_1^{-1}C_0^{-1}.\label{monform}\end{equation}
\end{lemma}

Рассмотрим проективизацию уравнения (\ref{linirr}): 
соответствующее уравнение Риккати на $\oc$-значную функцию 
$p(\tau)=\frac{z_2(\tau)}{z_1(\tau)}$, 
получаемое из (\ref{linirr}) с помощью тавтологической проекции 
$\cc^2\setminus0\to\oc=\mathbb{CP}^1$. 
В конце статьи мы будем использовать следующие свойства 
проективизированных канонических базисных решений уравнения (\ref{linirr}).

\begin{proposition} \label{canon-proj} Пусть уравнение (\ref{linirr})  аналитически эквивалентно 
своей формальной нормальной форме. Тогда проективизации 
его канонических базисных решений суть единственные решения $p(\tau)$ 
соответствующего уравнения Риккати, голоморфные в окрестности нуля. 
Их значения в нуле равны, соответственно, нулю и бесконечности в проективной 
координате $p=\frac{z_2}{z_1}$.  
\end{proposition}

Предложение \ref{canon-proj} хорошо известно специалистам и следует из его 
справедливости для формальной нормальной формы.

\begin{proposition} \label{proj-norm} В условиях предложения \ref{canon-proj} пусть 
$\psi_1(\tau)$, $\psi_2(\tau)$ --- рассматриваемые канонические решения уравнения 
Риккати, $\psi_1(0)=0$, $\psi_2(0)=\infty$, а 
$$z=H(\tau)w, \  H(\tau)=\left(\begin{matrix}& h_{11} & h_{12} \\
& h_{21} & h_{22}\end{matrix}\right)(\tau)$$
 --- нормализующая 
замена переменных, приводящая линейное уравнение (\ref{linirr})  к нормальной форме. 
Тогда 
\begin{equation}
\psi_j=\frac{h_{2j}}{h_{1j}}, \ j=1,2.\label{psi-h}\end{equation}
\end{proposition}

Предложение \ref{proj-norm} следует из диагональности канонической 
фундаментальной матрицы решений нормальной формы. 

\section{Число вращения и монодромия: доказательство теоремы \ref{th}}

\subsection{Целочисленность перемычек}
Рассмотрим уравнения  (\ref{jos}) класса Д и соответствующие им  
уравнения Риккати (\ref{ricc}) и линейные уравнения (\ref{lin}). Преобразования потока $h_{a,s}:S^1\to S^1$ продолжаются 
как мёбиусовы преобразования сферы Римана $\oc=\mathbb{CP}^1$ 
 и совпадают с 
преобразованием монодромии соответствующих уравнений Риккати 
при обходе вокруг точки $0\in\oc$, в силу предложения  \ref{prop-change}. 
Обозначим  
через $M_{a,s}$ оператор монодромии линейного уравнения (\ref{lin}). По 
определению, справедливо следующее 
\begin{proposition} \label{project} Продолженные 
преобразования потока $h_{a,s}$ совпадают с проективизациями 
операторов монодромии  $M_{a,s}$ соответствующих линейных уравнений 
(\ref{lin}). 
\end{proposition}

Росток каждого из линейных уравнений (\ref{lin}) в точке 0 иррегулярен и 
нерезонансен, точнее, имеет вид (\ref{linirr}). Его формальная нормальная 
форма имеет вид 
\begin{equation}\begin{cases} & \dot w_1=0 \\
& \dot w_2=\frac1{\tau^2}(\frac{is}2+a\tau)w_2\end{cases}.\label{norm-jos}\end{equation}
Оператор монодромии соответствующей формальной нормальной формы будем обозначать через $M_{N,a,s}$. Имеем: 
\begin{equation}M_{N,a,s}=\diag(1,e^{2\pi ia}).\label{monfjos}\end{equation}

\begin{proposition} \label{jos-mon-id} 
Точка $(a,s)$ отвечает перемычке, если и только если $s\neq0$ и 
соответствующий оператор монодромии $M_{a,s}$ проективно тождественен, т.е. 
есть произведение тождественного оператора на скалярный множитель.
\end{proposition}
\begin{proof} Точка отвечает перемычке, если и только если $s\neq0$ и  соответствующее преобразование потока уравнения класса Д 
за период тождественно, в силу предложения \ref{perem-id}. Это вместе с  
предложением \ref{project} доказывает предложение \ref{jos-mon-id}. 
\end{proof}
\begin{lemma} \label{anal-equiv} 
 Точка $(a,s)$, отвечает перемычке, если  и только если $a\in\zz$, 
$s\neq0$ и 
росток в нуле соответствующего линейного уравнения (\ref{lin}) аналитически 
эквивалентен его формальной нормальной форме (\ref{norm-jos}). 
\end{lemma}
\begin{proof} Пусть $s\neq0$, $a\in\mathbb Z$ и уравнение  (\ref{lin}) аналитически 
эквивалентно своей нормальной форме. Тогда точка $(a,s)$ отвечает 
перемычке, в силу 
предложения \ref{jos-mon-id} и тождественности монодромии (\ref{monfjos}). 
Докажем, что верно и обратное. Пусть точка $(a,s)$, $s\neq0$, отвечает перемычке.  Тогда оператор $M_{a,s}$ проективно тождественен, 
а значит, его матрица в любом базисе  должна иметь 
тривиальные верхне- и нижнетреугольные элементы. С другой стороны, 
его матрица в каноническом базисе пространства решений в секторе $S_0$, см. 
рис. \ref{fig:1}, есть произведение диагональной матрицы монодромии 
(\ref{monfjos}) формальной нормальной формы и двух матриц, обратным к 
матрицам Стокса. Последние две матрицы унипотентны, одна из них 
нижнетреугольна, а другая --- верхнетреугольна. Произведение трех 
вышеуказанных матриц имеет тривиальные треугольные элементы, если и 
только если две последние треугольные матрицы имеют тривиальные 
треугольные элементы, и тем самым, тождественны в силу унипотентности. 
Это в точности означает, что матрицы Стокса тождественны, и тем самым, 
линейное уравнение (\ref{lin}) аналитически эквивалентно своей формальной 
нормальной форме, по теореме  \ref{th-aneq}. В частности, монодромия 
уравнения (\ref{lin}) задается матрицей (\ref{monfjos}), которая проективно 
тождественна, если и только если $a\in\zz$. Итак, мы показали, что 
оператор монодромии уравнения (\ref{lin}) проективно тождественен, если и 
только если уравнение аналитически эквивалентно своей формальной 
нормальной форме и $a\in\zz$. Это вместе с предложением \ref{jos-mon-id} 
доказывает лемму \ref{anal-equiv}. Целочисленность абсцисс  перемычек 
доказана.
\end{proof}

\subsection{Случай $|\nu|\leq1$: перемычки $r$-го языка лежат на  прямой $a=r$}

Итак, мы показали, что все перемычки лежат на вертикальных прямых 
с целочисленными абсциссами. Теперь покажем, что при 
$|\nu|\leq1$ перемычки каждого 
индивидуального языка Арнольда лежат на одной и той же целочисленной 
вертикальной прямой, точнее,  
\begin{equation}\rho(a,s)=a \text{ для каждой перемычки } (a,s).\label{odna-pr}
\end{equation} 
\def\calc{\mathcal C}
Итак,  ниже считаем, что $|\nu|\leq1$. 
\begin{proposition} \label{nerav} При $|\nu|\leq1$ для любых $a,s\in\rr$ имеет место неравенство 
$a-1\leq\rho(a,s)\leq a+1$. Оба неравенства являются строгими, за исключением 
случая, когда $s=0$, $a=\pm1$. 
\end{proposition}
Неравенство из предложения \ref{nerav} содержится в статье \cite{buch1} 
и сформулировано там как лемма 4. Приведенное ниже доказательство 
строгости неравенства 
повторяет доказательство из статьи \cite{buch1} с 
небольшим дополнением.  

\begin{proof}  Всякое  решение $x(t)$ уравнения класса Д удовлетворяет 
дифференциальному неравенству 
\begin{equation} a+s\sin t-1\leq \dot x=\sin x+a+s\sin t\leq a+s\sin t+1.\label{ner-vo}
\end{equation}
В случае, когда $\sin x\not\equiv\pm1$, это неравенство является строгим при  почти всех $t$: при всех $t$, для которых $\sin x(t)\neq\pm1$.  
Значит, приращение решения $x(t)$ на 
любом отрезке $I$ длины $2\pi$ находится  между интегралами 
по тому же отрезку левой и правой частей предыдущего неравенства 
(см. теорему Чаплыгина о сравнении \cite{luz}).
Последние интегралы равны, соответственно, $a-1$ и $a+1$. Отсюда следует 
неравенство из предложения. Пусть теперь 
$\sin x(t)\not\equiv1$.  Тогда неравенство (\ref{ner-vo}) является 
строгим при почти всех $t$. Отсюда следует, что  предыдущее приращение 
находится строго между числами $a\pm1$ и равномерно отделено от них 
по всем отрезкам $I$ длины $2\pi$. Тем самым, число вращения также 
лежит строго между числами $a\pm1$, по определению и в силу 
равномерности. Пусть теперь неравенство (\ref{ner-vo}) не является строгим. 
Тогда в силу предыдущего, $\sin x(t)\equiv\pm1$, а значит, 
$x(t)\equiv\frac{\pi}2+\pi k$, $k\in\zz$. Уравнение (\ref{jos}) класса Д имеет 
постоянное решение $x(t)$ c $\sin x(t)\equiv\pm1$, если и только если $s=0$, 
$a=\mp1$: в этом случае число вращения равно нулю. 
Предложение доказано. 
\end{proof}

Рассмотрим язык Арнольда, отвечающий данному 
целому числу вращения $r$. Его дополнение до горизонтальной оси $s=0$ 
лежит строго между прямыми $a=r\pm1$, в силу предыдущего предложения. 
Его перемычки не лежат на горизонтальной оси (по определению) 
и имеют  целочисленные абсциссы, как показано выше. Значит, они лежат на 
прямой $a=r$. Первое утверждение теоремы \ref{th} доказано. 

\subsection{Случай произвольного $\nu$: модуль абсциссы перемычки не превосходит модуля числа вращения}

Здесь доказывается второе утверждение теоремы \ref{th}: при любом фиксированном $\nu\neq0$ абсцисса каждой перемычки имеет тот же знак, что 
и соответствующее число вращения, и ее модуль не превосходит модуля 
числа вращения.

Вначале мы докажем совпадение знаков, которое легко следует из того, 
что вся ось $s$ лежит в нулевом языке Арнольда (следующее предложение). Затем мы докажем вышесформулированное неравенство, используя комплексные 
уравнения Риккати и принцип аргумента для их канонических решений.

\begin{proposition} \label{propa0} При $a=0$ и любых $\nu$ и $s$ уравнение (\ref{jos}) 
имеет нулевое число вращения. Тем самым, при любом $\nu\neq0$ вся ось 
$s$ лежит в нулевом языке Арнольда. 
\end{proposition} 

\begin{proof} При $a=0$ уравнение (\ref{jos}) имеет симметрию 
$(x,t)\mapsto(-x,t+\pi)$. 

Тем самым, если $x(t)$ --- решение, то и  $\wt x(t)=-x(t+\pi)$ --- тоже решение. 
Отсюда следует, что инволюция $x\mapsto -x$ сопрягает 
отображения потока $H_{0,2\pi}:\rr\times\{0\}\to\rr\times \{0\}$  и 
$H_{\pi,2\pi}:\rr\times\{\pi\}\to\rr\times\{\pi\}$  поднятого уравнения (\ref{jos}). Тем самым, рассматриваемые  отображения  имеют противоположные числа вращения, 
определенные как пределы  (\ref{rot}). С другой стороны, их числа вращения 
равны числу вращения потока, см. раздел 2.1. Итак, 
рассматриваемые числа вращения одновременно равны и противоположны. 
 Следовательно, число вращения равно нулю. Предложение доказано.
\end{proof}

\begin{corollary} \label{cpropa0}
Абсцисса  каждой перемычки имеет тот же знак, что и 
соответствующее число вращения, если последнее не равно нулю. 
\end{corollary}

\begin{proof} Число вращения равно нулю при $a=0$ и 
не убывает как функция от $a$. Это доказывает следствие.
\end{proof}

\def\tt{\mathbb T}

Фиксируем $\nu\neq0$ и рассмотрим произвольную перемычку $(a,s)$.  Пусть 
$\rho=\rho(a,s)$ --- соответствующее число вращения. Напомним, 
что соответствующее линейное уравнение (\ref{lin}) аналитически эквивалентно 
своей формальной нормальной форме в нуле (лемма \ref{anal-equiv}). Отсюда 
следует, что проективизированные 
канонические базисные решения уравнения (\ref{lin}) 
голоморфны в полной окрестности 
нуля (предложение \ref{canon-proj}), а следовательно, на всей прямой $\cc$. 
Обозначим последние проективизированные решения 
через $\psi_1(\tau)$, $\psi_2(\tau)$, так что $\psi_1(0)=0$, $\psi_2(0)=\infty$, см. 
предложение \ref{canon-proj}. Положим 
$$S^1_p=\{ |p|=1\}\subset\oc, \ S^1_{\tau}=\{|\tau|=1\}\subset\oc, \ \tt^2=S^1_p\times S^1_{\tau}, \ D_1=\{|\tau|<1\}.$$
\begin{remark} \label{psineq} Тор $\tt^2$ инвариантен относительно уравнения Риккати (\ref{ricc}), ограниченного на $\oc\times S^1_{\tau}$. 
Функции $\psi_1$, $\psi_2$ суть его различные решения, 
а значит, $\psi_1(\tau)\neq\psi_2(\tau)$ 
при всех $\tau\in\cc$. Ограничение каждой из них 
на единичную окружность $S^1_{\tau}$ принимает значения, лежащие 
либо все по одну сторону от единичной окружности $S^1_p$,  либо на самой 
окружности. Это следует из инвариантности тора $\tt^2$.
\end{remark}

Доказательство неравенства $|a|\leq|\rho|$ для перемычек 
из теоремы \ref{th} основано на 
следующей  формуле для числа вращения $\rho$.

\begin{lemma} \label{prop-rho-form}
 Пусть $|\psi_1||_{S^1_{\tau}}\leq1$. Тогда 
\begin{equation}\rho=a-2\#(\text{полюсов функции } \psi_1|_{D_1}).\label{rho-pol}
\end{equation}
\end{lemma}

\begin{proof} Для произвольного решения $\psi$ уравнения Риккати 
(\ref{ricc}) и любого $r>0$, такого что  $\psi|_{\{|\tau|=r\}}\neq\infty$, рассмотрим 
уравнение в вариациях вдоль решения $\psi$. Заметим, что решение 
уравнения в вариациях однозначно определено с точностью до 
мультипликативной константы и голоморфно по $\tau$ в окрестности 
окружности $\{|\tau|=r\}$. Это следует из того, что уравнение Риккати имеет 
тривиальную монодромию, как и соответствующее линейное уравнение 
(предложение \ref{jos-mon-id}). 
Фиксируем произвольное ненулевое решение $v(\tau)$ уравнения в вариациях. 
 Ориентируем окружность $\{|\tau|=r\}$ 
против часовой стрелки и введем индекс 
$$\chi_r(\psi)=\text{ приращение аргумента решения } v(\tau) 
\text{ вдоль окружности }
\{|\tau|=r\}.$$

\begin{proposition} \label{prop-rho0}
 Для любого достаточно малого $r>0$  \ $\chi_r(\psi_1)=a$. 
\end{proposition}

Справедливость предложения следует из его справедливости для уравнения 
Риккати, являющегося проективизацией формальной нормальной формы 
(\ref{norm-jos}), 
и аналитической нормализуемости рассматриваемого линейного уравнения 
(\ref{lin}). 

\begin{proposition} \label{prop-rho1}
Пусть $|\psi_1||_{S^1_{\tau}}\leq1$. Тогда $\chi_1(\psi_1)=\rho$. 
\end{proposition}

\begin{proof} Фиксируем $\tau_0\in S^1_{\tau}$. Преобразования 
потока $\oc\times\{\tau_0\}\to\oc\times\{\tau\}$ уравнения Риккати вдоль единичной окружности $|\tau|=1$  суть мёбиусовы преобразования, сохраняющие 
единичный диск. Индекс $\chi_1(\psi)$ любого решения уравнения 
Риккати с $\psi(\tau_0)\in \overline{D_1}$ не зависит от выбора решения и равен 
$\rho$ при $\psi(\tau_0)\in S^1_p$. Это доказывает предложение.
\end{proof}

\begin{proposition} \label{ind-pol}
 Рассмотрим произвольное уравнение Риккати. Пусть 
$0<r_1<r_2$ --- такие, что уравнение Риккати голоморфно 
в окрестности замкнутого  кольца 
$A_{r_1r_2}=\{ r_1\leq |\tau|\leq r_2\}$ и имеет там мероморфное 
решение $\psi(\tau)$ без полюсов на границе кольца. Тогда   
$$\chi_{r_2}(\psi)=\chi_{r_1}(\psi)-
2\#(\text{полюсов функции } \psi|_{r_1<|\tau|<r_2}).$$
\end{proposition}

\begin{proof} 
Ненулевое решение $v(\tau)$ уравнения в вариациях вдоль $\psi$ 
есть функция 
со  значениями в  касательном расслоении к сфере Римана, 
голоморфная в окрестности  кольца $A_{r_1r_2}$ и 
не обращающаяся в нуль.  Пусть $p$ --- 
комплексная координата на $\cc$. Она задает стандартную тривиализацию 
касательного расслоения $T\oc$ над $\cc$. Функция $v(\tau)$ со значениями в 
$T\oc$, записанная в стандартной тривиализации, мероморфна, и мы будем ее 
обозначать тем же символом $v(\tau)$. Покажем, что функция 
$v(\tau)$ имеет те же полюса и той же кратности, что и функция $\psi^2(\tau)$. 
Это вместе с принципом аргумента докажет предложение \ref{ind-pol}.  
Пусть $\tau'\in Int(A_{r_1r_2})$ --- полюс функции $\psi$. 
Пусть $V\subset\oc\setminus0$ --- произвольная малая окрестность бесконечности, 
точнее, дополнение к большому диску с центром в нуле, а   
$D\subset A_{r_1r_2}$ --- малый замкнутый диск с центром в $\tau'$, такой что 
$\psi(D)\subset V$. Рассмотрим вспомагательную тривиализацию касательного 
расслоения $T\oc$ над $V$, заданную с помошью карты $\wt p=\frac1p$. Напомним, 
что  $T\oc$- значное решение уравнения в вариациях голоморфно на замкнутом диске $D$ и в новой тривиализации записывается как 
$\wt v(\tau)=\psi^{-2}(\tau)v(\tau)$. По построению, функция $\wt v(\tau)$  
голоморфна и не обращается в нуль на диске $D$, и $v(\tau)=\wt v(\tau)\psi^2(\tau)$. 
 Следовательно, функция $v(\tau)$ имеет те же полюса и той же кратности, 
 что и функция $\psi^2(\tau)$.  Предложение \ref{ind-pol} доказано.
\end{proof}

Справедливость леммы \ref{prop-rho-form} вытекает из 
предложений \ref{prop-rho0}, \ref{prop-rho1} и \ref{ind-pol}.
\end{proof}

\begin{lemma} \label{lem-razn} 
Пусть функции $\psi_1|_{S^1_{\tau}}$ и $\psi_2|_{S^1_{\tau}}$ 
либо принимают значения,  лежащие по разные стороны от единичной 
окружности $S^1_p$, либо $\psi_j(S^1_{\tau})\subset S_p^1$ при 
некотором $j=1,2$. Тогда число вращения рассматриваемого уравнения класса 
Д равно абсциссе перемычки: $\rho=a$. 
\end{lemma}
Ниже мы выведем лемму \ref{lem-razn} из леммы \ref{prop-rho-form} и 
следующего предложения.

\begin{proposition} \label{prop-razn} В условиях леммы \ref{lem-razn} 
\begin{equation}
|\psi_1||_{\overline{D_1}}\leq 1, \ |\psi_2||_{\overline{D_1}}\geq1.\label{psi1}
\end{equation}
\end{proposition}

\begin{proof} Рассмотрим функцию
$$\phi=\frac{\psi_1}{\psi_2}.$$
Функция $\phi(\tau)$ мероморфна на $\cc$ и обладает следующими свойствами:
\begin{equation} \phi\neq1, \ \phi(0)=0; \ \text{ либо } |\phi||_{S^1_{\tau}}\leq1, 
\ \text{ либо } |\phi||_{S^1_{\tau}}>1.\label{prophi}\end{equation}
Утверждение (\ref{prophi}) следует из условий леммы  \ref{lem-razn} и 
замечания \ref{psineq}. Мы разберем вначале первый случай, когда 
$|\phi||_{S^1_{\tau}}\leq1$, а затем покажем, что второй случай, когда 
$|\phi||_{S^1_{\tau}}>1$, невозможен. 
Рассмотрим разность $\phi(\tau)-1$. Она мероморфна и не обращается 
в нуль на $\cc$, а приращение ее аргумента  вдоль 
окружности $S^1_{\tau}$  равно нулю, так как 
$\phi(S^1_{\tau})\subset\overline{D_1}$. Следовательно, функция $\phi-1$ 
не имеет полюсов в единичном диске. Значит, $|\phi||_{D_1}\leq1$, по 
принципу максимума и так как последнее неравенство выполнено на границе 
диска по предположению. Итак,  $|\psi_1|\leq|\psi_2|$ на единичном диске. 
С другой стороны, на 
единичной окружности выполнено одно из неравенств 
$|\psi_1|\leq 1\leq|\psi_2|$, $|\psi_2|\leq 1\leq|\psi_1|$, в силу условия леммы 
 \ref{lem-razn}.  
Следовательно,  первое неравенство выполнено на всем замкнутом 
единичном диске, и предложение \ref{prop-razn} доказано. Пусть теперь имеет 
место второй случай: $|\phi||_{S^1_{\tau}}>1$, т.е. $|\phi^{-1}||_{S^1_{\tau}}<1$.  
Разность $\phi^{-1}(\tau)-1$ мероморфна и не обращается в нуль на $\cc$, а 
приращение ее аргумента  вдоль окружности $S^1_{\tau}$  равно нулю, 
как и выше. В то же время, рассматриваемая разность имеет полюс в нуле, что 
противоречит принципу аргумента. Полученное противоречие доказывает 
невозможность 
рассматриваемого случая и завершает доказательство предложения 
\ref{prop-razn}. 
\end{proof}

\begin{proof} {\bf леммы \ref{lem-razn}.} Имеем $|\psi_1|_{D_1}\leq1$, 
в силу предложения \ref{prop-razn}.  Отсюда и из леммы \ref{prop-rho-form} 
следует утверждение леммы \ref{lem-razn}.
\end{proof}

\begin{lemma} \label{same-side} Пусть на единичной окружности $S^1_{\tau}$ 
выполнены неравенства $|\psi_1|,|\psi_2|\leq1$. Тогда 
 $\rho\leq a<0$, где 
$a$ --- абсцисса рассматриваемой перемычки. 
\end{lemma}

\begin{proof} Неравенство $\rho\leq a$  следует из леммы \ref{prop-rho-form}. 
Достаточно показать, что $\rho<0$: тогда $a<0$, в силу следствия \ref{cpropa0}.

\begin{proposition} \label{psi's} В условиях леммы \ref{same-side} число 
вращения $\rho$ равно приращению аргумента разности  
$\psi_2-\psi_1$ вдоль единичной окружности. 
\end{proposition}

\begin{proof} Напомним, что $\psi_1(\tau)\neq\psi_2(\tau)$ при всех $\tau\in\cc$. 
Рассмотрим отображения потока $H_{1,\tau}:  
\oc\times\{1\}\to\oc\times\{\tau\}$ уравнения Риккати (\ref{ricc}) 
при $\tau\in S^1_{\tau}$. Преобразования $H_{1,\tau}$ являются 
конформными автоморфизмами единичного диска, и 
$H_{1,\tau}(\psi_j(1),1)=(\psi_j(\tau),\tau)$ при $j=1,2$, так как $\psi_j$ суть решения 
 уравнения (\ref{ricc}). 
Отсюда следует, что геодезическая  метрики Пуанкаре на $D_1$, соединяющая 
точки $\psi_1(1)$ и $\psi_2(1)$, переходит в геодезическую $\gamma_{\tau}$, соединяющую точки 
$\psi_1(\tau)$ и $\psi_2(\tau)$. Каждый из концов геодезической $\gamma_{\tau}$ пробегает 
единичную окружность, делая $\rho$ полных оборотов, когда $\tau$ 
пробегает единичную окружность $S^1_{\tau}$ в положительном направлении. 
Это следует из определения числа вращения. С другой стороны, приращение 
аргумента направляющего единичного 
вектора геодезической $\gamma(\tau)$ в точке $\psi_1(\tau)$ равно приращению 
аргумента прямолинейного вектора на Евклидовой плоскости 
$\rr^2=\cc\subset\oc$, 
 направленного из точки $\psi_1(\tau)$ в точку $\psi_2(\tau)$. Это следует 
 из того, что  угол между 
 рассматриваемыми векторами меньше $\pi$ при всех $\tau$. Итак, приращение 
 аргумента функции $\psi_2(\tau)-\psi_1(\tau)$ равно числу вращения. 
 Предложение \ref{psi's} доказано.
\end{proof}

\begin{proposition} \label{arg-razn} В условиях леммы \ref{same-side} приращение 
аргумента  разности $\psi_2-\psi_1$ вдоль единичной окружности отрицательно.
\end{proposition}
\begin{proof} Рассмотрим матричнозначную функцию $H(\tau)$ нормализующей 
замены, приводящей линейное уравнение (\ref{lin}) к нормальной форме 
(\ref{norm-jos}): $z=H(\tau)w$. Матричнозначная функция $H(\tau)$ голоморфна 
на $\cc$ и принимает значения в обратимых матрицах. 
Напомним, что $\psi_j=\frac{h_{2j}}{h_{1j}}$ при $j=1,2$, в силу 
предложения \ref{proj-norm}. Тем самым, 
$$\psi_2-\psi_1=\frac{h_{11}h_{22}-h_{21}h_{12}}{h_{11}h_{12}}.$$
Числитель в последней правой части есть определитель обратимой 
матричнозначной функции $H$,  а значит, 
он не обращается в нуль на $\cc$. Тем самым,  приращение аргумента 
рассматриваемой разности 
равно минус сумме количеств нулей (с учетом кратностей) функций $h_{11}$ и 
$h_{12}$ в единичном диске. Последняя сумма положительна, так как 
$h_{12}(0)=0$.  
Следовательно, рассматриваемое приращение аргумента отрицательно. 
Предложение \ref{arg-razn} доказано.
\end{proof}

Неравенство $\rho<0$  следует из предложений \ref{psi's} и \ref{arg-razn}. 
Это вместе с рассуждениями из начала доказательства леммы \ref{same-side} 
доказывает лемму  \ref{same-side}. 
\end{proof}

{\bf Окончание доказательства теоремы \ref{th}.} Мы уже доказали, 
что абсциссы перемычек целочисленны. Напомним, что для каждой перемычки 
 каждая из соответствующих 
функций $\psi_j|_{S^1_{\tau}}$, $j=1,2$, принимает значения либо по одну сторону 
от единичной окружности, либо на самой окружности. Имеет место один из следующих трех случаев: 

1) функции $\psi_j|_{S^1_{\tau}}$, $j=1,2$, принимают значения, лежащие 
по разную сторону от единичной окружности; сюда мы включаем и 
случай, когда $|\psi_j||_{S^1_{\tau}}\equiv1$ при некотором $j=1,2$; 

2)  $|\psi_j||_{S^1_{\tau}}<1$ при всех  $j=1,2$;

3)  $|\psi_j||_{S^1_{\tau}}>1$ при всех  $j=1,2$. 

В первом случае имеем $\rho=a$, по лемме \ref{lem-razn}. Во втором случае 
имеем $\rho\leq a<0$, по лемме \ref{same-side}. Третий случай сводится ко 
 второму заменой переменных $(x,t)\mapsto(-x,t+\pi)$ 
 ($(p,\tau)\mapsto (p^{-1},-\tau)$ в координатах $(p,\tau)$). Последняя замена 
  переводит уравнение класса Д в то же уравнение, но меняет знак параметра 
  $a$ и числа вращения. Отсюда и из предыдущего утверждения следует, что 
  в третьем случае имеет место неравенство $0< a\leq\rho$. При 
$a=0$ имеем $\rho=0$, в силу предложения \ref{propa0}. 
Это доказывает теорему \ref{th}.

\subsection{Экспериментальный факт А в общем случае: текущее состояние дел}

Экспериментальный факт А, утверждающий, что  $\rho=a$ для каждой перемычки,  доказан при $|\nu|\leq1$ (теорема \ref{th}). Ожидается, что 
он верен при любом $\nu\neq0$. 

\begin{theorem} \label{final} Для каждой перемычки разность  $\rho-a$ всегда четна. 
Если она не равна нулю, то либо $\rho<a<0$, либо $0<a<\rho$. 
\end{theorem}

Теорема \ref{final} следует  из леммы \ref{prop-rho-form} и теоремы \ref{th}. 

Пусть $(\nu,a,s)$ --- перемычка. Пусть 
$\psi_1$, $\psi_2$ --- голоморфные в нуле решения соответствующего  уравнения Риккати (\ref{ricc}) из предложения \ref{proj-norm}. 

Как показано ниже, равенство $\rho=a$ имеет место при следующем  
условии:

{\bf  Условие $(*)$ на уравнение Риккати (\ref{ricc}):}    либо $\psi_1(\tau)$ не 
имеет полюсов в единичном диске и $|\psi_1||_{S^1_{\tau}}\leq1$, либо 
$\psi_2(\tau)$ не имеет нулей в единичном диске и $|\psi_2||_{S^1_{\tau}}\geq1$.

\begin{lemma} \label{lfinal} Пусть $(\nu,a,s)$ --- перемычка ($s\neq0$).  
Тогда $\rho=a$, если и только если 
соответствующее уравнение Риккати удовлетворяет условию $(*)$.  
\end{lemma} 

Лемма \ref{lfinal} следует из леммы \ref{prop-rho-form}, аналог которой 
справедлив и в случае, когда $|\psi_2||_{S^1_{\tau}}\geq1$, с заменой функции 
$\psi_1$ на $\psi_2$, а  количества полюсов --- на  количество нулей. 

\medskip

{\bf Гипотеза С.} Условие $(*)$ выполнено для всякой перемычки.

\medskip

Экспериментальный факт А эквивалентен гипотезе С, в силу леммы \ref{lfinal}. 

\bigskip

Авторы благодарны В.М.Бухштаберу, Э.Жису, Ю.С.Ильяшенко, А.В.Клименко и 
О.Л.Ромаскевич за полезные обсуждения.